\documentclass[letter, 12pt]{article}

\usepackage{amscd}
\usepackage{amsmath}
\usepackage{amssymb}
\usepackage{amsthm}
\usepackage{fancyhdr}
\usepackage{verbatim}
\usepackage{tikz}
\usetikzlibrary{snakes}
\usepackage[all]{xy}
\usepackage{color}
\usepackage[colorlinks=true, linkcolor=blue, citecolor=blue, pagebackref=true]{hyperref}
\usepackage[pagebackref=true]{hyperref}

\newtheorem{theorem}{Theorem}[section]
\newtheorem{lemma}[theorem]{Lemma}
\newtheorem{proposition}[theorem]{Proposition}

\theoremstyle{plain}

\theoremstyle{definition}

\theoremstyle{definition} 
\newtheorem*{definition*}{Definition}

\theoremstyle{example} 
\newtheorem*{example*}{Example}

\theoremstyle{theorem}

\theoremstyle{theorem} 
\newtheorem*{conjecture*}{Conjecture}

\theoremstyle{theorem}

\theoremstyle{theorem} 
\newtheorem{assumption}[theorem]{Assumption}

\theoremstyle{theorem} 
\newtheorem*{assumption*}{Assumption}

\theoremstyle{remark}

\theoremstyle{remark} 
\newtheorem*{remark*}{Remark}

\newcommand{\NN}{\mathbb{N}}
\newcommand{\RR}{\mathbb{R}}
\newcommand{\CC}{\mathbb{C}}

\newcommand{\ZZ}{\mathbb{Z}}
\newcommand{\LLL}{\mathcal{L}}

\newcommand{\D}{\mathcal{D}}
\newcommand{\E}{\mathcal{E}}

\newcommand{\U}{\mathcal{U}}

\newcommand{\Z}{\mathcal{Z}}

\newcommand{\W}{\mathcal{W}}

\newcommand{\di}{\mathrm{d}}

\newcommand{\bk}{\boldsymbol{k}}

\newcommand{\bell}{\boldsymbol{\ell}}
\newcommand{\bm}{\boldsymbol{m}}

\newcommand{\bnu}{\boldsymbol{\nu}}

\newcommand{\bZ}{\boldsymbol{\mathcal{Z}}}

\newcommand{\Sl}{\mathfrak{sl}}
\newcommand{\SL}{\mathrm{SL}}
\newcommand{\SO}{\mathrm{SO}}

\newcommand{\PSL}{\mathrm{PSL}}

\newcommand{\Cinf}{C^{\infty}}
\newcommand{\Lii}{L^{2}}

\newcommand{\HH}{\mathcal{H}}
\newcommand{\II}{\mathcal{I}}

\newcommand{\cee}{c}

\providecommand{\abs}[1]{\lvert#1\rvert}
\providecommand{\Abs}[1]{\left|#1\right|}
\providecommand{\norm}[1]{\lVert#1\rVert}
\providecommand{\Norm}[1]{\left\|#1\right\|}

\def\d{{\delta}}
\def\e{{\epsilon}}
\def\s{{\sigma}}
\def\vars{{\varsigma}}

\def\t{{\tau}}

\def\l{{\lambda}}
\def\L{{\Lambda}}
\def\o{{\omega}}
\def\O{{\Omega}}
\def\G{{\Gamma}}

\title{\textbf{Higher cohomology of parabolic actions on certain homogeneous spaces}}
\author{Felipe A.~Ram{\'i}rez\footnote{\textbf{email:} \texttt{felipe.ramirez@york.ac.uk}} \\ University of York}
\date{}

\begin{document}

\maketitle
\thispagestyle{empty}


\begin{abstract}
We show that for a parabolic $\RR^d$-action on $\PSL(2,\RR)^d/\G$, the cohomologies in degrees $1$ through $d-1$ trivialize, and we give the obstructions to solving the degree-$d$ coboundary equation, along with bounds on Sobolev norms of primitives. In previous papers we have established these results for certain Anosov systems. The present work extends the methods of those papers to systems that are \emph{not} Anosov. The main new idea is in \S\ref{sec:topdegree}, where we define special elements of representation spaces that allow us to modify the arguments from the previous papers. In \S\ref{sec:discussion} we discuss how one may generalize this strategy to $\RR^d$-systems coming from a product of Lie groups, like in the systems we have here.
\end{abstract}


\newpage
{\footnotesize\tableofcontents}

\section{Introduction}

This article is a complement to~\cite{Ramhc}, where we studied the smooth cohomology of Anosov $\RR^d$-actions on homogeneous spaces of the $d$-fold product $\SL(2,\RR)^d=\SL(2,\RR)\times\dots\times\SL(2,\RR)$. The main results there are a description of top-degree (degree $d$) cohomology and a vanishing statement for the lower degrees (degrees $1,\dots,d-1$). Here we consider \emph{unipotent} $\RR^d$-actions on irreducible compact quotients of $\PSL(2,\RR)^d$, and prove similar statements to those we proved in the Anosov case. Unlike in the Anosov case, these results are not expected to hold for all unipotent systems. In fact, they are in a sense optimal for our examples. We comment on this after stating Theorem~\ref{thm:unipotent}. 

\subsection*{Cohomology in dynamics}
Cohomology is a fundamental tool in the study of rigidity properties of dynamical systems. For some background on this, we recommend~\cite{Fur81, KR, KN11}. Let us briefly review some definitions that are relevant to this paper. (The initiated reader may skip to the ``Past work'' section.)

For a flow on a manifold $M$ along a vector field $V$, the \textbf{(degree-$1$) coboundary equation} is the familiar $V\,g=f$, asking us to determine whether a given smooth function (or \textbf{$1$-cocycle}) $f\in\Cinf(M)$ is the derivative of some other smooth function $g\in\Cinf(M)$ in the flow direction. There is a host of literature, even on this seemingly modest question. For now we only mention the famous Livshitz Theorem, which gives a beautiful answer in the case of an Anosov flow on a compact manifold. In words, the theorem states that we only need to check that $f$ has integral $0$ around every periodic orbit of the flow~\cite{Livsic, GK80, GK80b, dlLMM}. (Though we refer to it as the Livshitz Theorem, the full statement, with regularity of solutions and all, is really the culmination of work by several authors.)

For an $\RR^d$-action generated by commuting vector fields $V_1,\dots,V_d$ on a manifold $M$, one can define cohomology in degrees $1,\dots,d$. The \textbf{degree-$d$ coboundary equation} is
\begin{equation}\label{eqn:dcobeq}
	V_1\,g_1+\dots+V_d\,g_d = f.
\end{equation}
This corresponds to asking whether a closed orbit-leafwise differential $d$-form (or \textbf{$d$-cocycle}, the one determined by $f\in\Cinf(M)$) is a coboundary, in the sense of having a \textbf{$(d-1)$-primitive} defined by $d$ smooth functions $g_1,\dots,g_d \in \Cinf(M)$.  Accordingly, there are coboundary equations in all degrees, corresponding to the usual \textbf{coboundary equation} $\di\eta=\o$ for leafwise differential forms, asking us to find an $(n-1)$-form $\eta \in \O^{n-1}$ whose leafwise exterior derivative is the given leafwise $n$-cocycle $\o \in \O^n$. The leafwise exterior derivative $\di$ is defined by the formula
\[
\di\o(V_{i_1}, \dots, V_{i_{n+1}}) = \sum_{j=1}^{n+1}(-1)^{j+1}V_{i_j}\,\o(V_{i_1},\dots, \widehat{V_{i_j}}, \dots, V_{i_{n+1}}),
\]
where $1\leq i_1<\dots<i_{n+1}\leq d$ and $\widehat{V_{i_j}}$ means that $V_{i_j}$ is omitted. So in degrees $1,\dots, d-1$, the coboundary equation is a system of $\binom{d}{n}$ partial differential equations, instead of just one ($1=\binom{d}{d}$) differential equation as in~\eqref{eqn:dcobeq}.

\subsection*{Past work}
There is a conjectural generalization of the Livshitz Theorem to higher-rank Anosov actions, due to A.~and S.~Katok~\cite{KK95}. They predict that in order for the degree-$d$ coboundary equation to have a smooth solution, one should only have to check that $f$ integrates to $0$ over closed orbits of the Anosov $\RR^d$-action. Furthermore, when $d \geq 2$, all lower cohomologies should trivialize. (It is already known by work of A.~Katok and R.~Spatzier that the \emph{first} cohomology does~\cite{KS94first}.) Katok and Katok proved their conjecture for partially hyperbolic $\ZZ^d$-actions on tori~\cite{KK95,KK05}. 

This problem was the motivation for our previous papers on higher cohomology~\cite{Ramhc,Ramhcii}. There, we studied certain families of \emph{Weyl chamber flows}, and proved that indeed the lower cohomologies trivialize for those systems, and that in degree $d$---the \emph{top} degree---one only has to deal with obstructions coming from action-invariant distributions. Integration over a closed orbit is itself an invariant distribution, so our results in top degree are \emph{a priori} weaker than the expected statements, but they are a step in the right direction. 

\subsection*{The present work}
In this article we work with \emph{unipotent} systems, where there may not even be any closed orbits. Yet, we know from work of L.~Flaminio and G.~Forni~\cite{FF} that for the horocycle flow of a hyperbolic surface, there are always obstructions to the degree-$1$ coboundary equation, coming from flow-invariant distributions. Therefore, a statement like the Katok--Katok Conjecture cannot possibly be true in the unipotent case, at least not the top-degree part.

On the other hand, it is natural to expect that the obstructions to the top-degree coboundary equation for a unipotent $\RR^d$-action come from invariant distributions, as in the work of Flaminio--Forni. Indeed, we prove the following result for the cohomology of unipotent $\RR^d$-actions on homogeneous spaces of the $d$-fold product $\PSL(2,\RR)^d$.

\begin{theorem}\label{thm:unipotent}
Let  $G= G_1\times\dots\times G_d$ with $G_i \cong \PSL(2,\RR)$ for $i=1,\dots, d$,  and let $\G \subset G$ be an irreducible cocompact lattice. Consider the maximal unipotent $\RR^d$-action on $G/\G$. For any smooth member of the kernel of all $\RR^d$-invariant distributions there is a smooth solution to the degree-$d$ coboundary equation. Also, the cohomologies in degrees $1,\dots, d-1$ trivialize, meaning that any smooth cocycle is smoothly cohomologous to one defined by constant functions on $G/\G$.
\end{theorem}


\begin{remark*}
We in fact prove versions of Theorem~\ref{thm:unipotent} that also give bounds on Sobolev norms of solutions to coboundary equations, listed here as Theorems~\ref{thm:sobtop} and~\ref{thm:soblower}.
\end{remark*}

\subsection*{Expectations}
Notice that since there are no closed orbits in the systems of Theorem~\ref{thm:unipotent}, a statement of this type is the best one can hope for in terms of describing the obstructions to solving the top-degree coboundary equation. On the other hand, not even the \emph{lower}-degree part of Theorem~\ref{thm:unipotent} holds for every unipotent system: Already, there are non-vanishing obstructions (to trivialization of first cohomology) for unipotent $\RR^2$-actions on homogeneous spaces of $\SL(2,\CC)$ found by Mieczkowski~\cite{M1}; Wang~\cite{Wan12} has shown the same for unipotent $\RR^2$-actions on homogeneous spaces of $\SL(3,\RR)$; and, in an ongoing project, L.~Flaminio and the author show that there are non-vanishing obstructions to cohomology of the \emph{horospheric} action in infinitely many irreducible unitary representations of $\SO^\circ(1,N)$, for any $N\geq 2$. 

\subsection*{Methods}
The arguments used here are adaptations of the representation-theoretic methods used in~\cite{Ramhc,Ramhcii}. The main new idea here is in defining ``$\varphi$'s'' (see \S\ref{sec:topdegree}), or, more precisely, articulating them. These are special elements of irreducible unitary representations of $\PSL(2,\RR)$ that facilitate an inductive argument in the proof of the main theorem. After reinterpretation, one can say that similar special elements were also present in the previous papers, but were ``hidden,'' so went without notice. By drawing attention to them we are now able to modify the procedures from our previous articles (which were for Anosov systems) so that they work for the unipotent systems in this article. Our \S\ref{sec:discussion} is a discussion of this, and how it may generalize.

\paragraph{}
By combining Theorem~\ref{thm:unipotent} with the results in~\cite{Ramhc}, we realize the following result for ``mixed'' actions.
\begin{theorem}\label{thm:mixed}
Let  $G= G_1\times\dots\times G_d$ with each $G_i \cong \PSL(2,\RR)$ and let $\G \subset G$ be an irreducible cocompact lattice. Consider an $\RR^d$-action on $G/\G$ whose projection to each factor corresponds to either the geodesic or horocycle flow. Then for any smooth member of the kernel of all $\RR^d$-invariant distributions there is a smooth solution to the degree-$d$ coboundary equation. Also, the cohomologies in degrees $1,\dots, d-1$ trivialize.
\end{theorem}


\section{Results}

Our main results are stated for unitary representations of $\PSL(2,\RR)^d$ that satisfy the following assumption, which in particular holds for the regular representation of $\PSL(2,\RR)^d$ on $\Lii(\PSL(2,\RR)^d/\G)$ where $\G\subset\PSL(2,\RR)^d$ is an irreducible cocompact lattice.

\begin{assumption}\label{ass:ume}
For the unitary representation $\PSL(2,\RR)^d\to\U(\HH)$ there exists $\e_0 >0$ and direct integral decomposition
\[
	\HH = \int_\RR^\oplus \HH_\l\,ds(\l)
\]
such that for $ds$-almost every $\l$,
\[
	\HH_\l = \HH_{\nu_1(\l)}\otimes\dots\otimes\HH_{\nu_d(\l)}
\]
where $\nu_j (\l) \notin B(0,\e_0)\backslash\{0\}$ for all $j=1,\dots,d$. Colloquially, ``none of the $\nu_j$'s accumulate to zero.''
\end{assumption}

The following two theorems are the main results of this paper. They constitute a version of Theorem~\ref{thm:unipotent} for unitary representations of $\PSL(2,\RR)^d$ that satisfy Assumption~\ref{ass:ume} and have spectral gap. The statements include estimates on Sobolev norms.

\begin{theorem}[Version of Theorem~\ref{thm:unipotent} for unitary representations with spectral gap; top degree]\label{thm:sobtop}
Let $\HH$ be the Hilbert space of a unitary representation of $\PSL(2,\RR)^d$ with a spectral gap and satisfying Assumption~\ref{ass:ume}. Suppose $f \in \Cinf(\HH)$ lies in the kernel of all $U_1,\dots,U_d$-invariant distributions. Then there exists a solution $g_1,\dots,g_d \in \Cinf(\HH)$ to the degree-$d$ coboundary equation for $f$ that satisfies the Sobolev estimates $\norm{g_i}_t \ll_{\nu_0, t} \norm{f}_{\s_d}$, where $\s_d := \s_d (t)$ is some increasing function of $t>0$.

\end{theorem}

\begin{theorem}[Version of Theorem~\ref{thm:unipotent} for unitary representations with spectral gap; lower degrees]\label{thm:soblower}
Let $\HH$ be the Hilbert space of a unitary representation of $\PSL(2,\RR)^d$ with spectral gap and satisfying Assumption~\ref{ass:ume}, and suppose that almost every irreducible representation appearing in its direct decomposition has no trivial factor. Then for $1\leq n\leq d-1$, any smooth $n$-cocycle $\o\in\O_{\RR^d}^n(\Cinf(\HH))$ is the coboundary $\di\eta=\o$ of some $\eta\in\O_{\RR^d}^{n-1}(\Cinf(\HH))$ and $\norm{\eta}_t \ll_{\nu_0,t}\,\norm{\o}_{\vars_d}$ where $\vars_d:=\vars_d(t)$ is an increasing function of $t>0$.
\end{theorem}

\section{Preliminaries}

Let $U = \left(\begin{smallmatrix} 0 & 1 \\ 0 & 0\end{smallmatrix}\right)\subset \Sl(2,\RR)$ be the generator of the horocycle flow. We will take our unipotent $\RR^d$-action to be that generated by $U_i:=U$ coming from each factor of the $d$-fold sum $\Sl(2,\RR)\oplus\dots\oplus\Sl(2,\RR)$.

\subsection*{Irreducible unitary representations}
We work with irreducible unitary representations $\pi_\nu: \PSL(2,\RR)\to\U(\HH_\nu)$ and $\pi_{\bnu}: \PSL(2,\RR)^d\to\U(\HH_{\bnu})$, where $\nu$ is a parameter taking the values
\[
\nu \in \begin{cases}
	i\,\RR &\textrm{principal series representations} \\
	(-1,1)\backslash\{0\} &\textrm{complementary series representations} \\
	2\NN-1 &\textrm{discrete series representations}
\end{cases}
\]
and $\HH_{\bnu}:=\HH_{\nu_1}\otimes\dots\otimes\HH_{\nu_d}$. Each $\HH_\nu$ has a basis $\{u(k)\}_{k\in\Z_\nu}$, where 
\[
	\Z_\nu = \begin{cases}
		\ZZ &\textrm{if $\HH_\nu$ is from the principal or complementary series} \\
		\ZZ_{\geq 0}+n &\textrm{if $\nu=2n-1$, $\HH_\nu$ from discrete series.}
	\end{cases}
\]
Any element $f \in \HH_\nu$ can therefore be written in terms of its coefficients $f=\sum_{k\in\ZZ_\nu} f(k)\,u(k)$. The basis $\{u(k)\}$ consists of eigenvectors for the Laplacian, so Sobolev norms can be easily computed in terms of coefficients. We have $f \in W^t(\HH_\nu)$ if and only if 
\[
	\norm{f}_t^2 = \sum_{k\in\Z_\nu} (1+\mu+2k^2)^t\,\abs{f(k)}^2\,\norm{u(k)}^2 < \infty,
\]
where $\mu$ is the eigenvalue of the Casimir operator $\Box$ corresponding to $\nu$ through $\nu^2 = 1-4\mu$. For convenience, we put $Q_\nu (k) = \mu + 2k^2$. 

We will use the following lemma.

\begin{lemma}\label{lem:diffnorm}
Let $\PSL(2,\RR)\to\U(\HH)$ be a unitary representation with direct decomposition 
\[
	\HH = \int_\RR^\oplus \HH_\l\,ds(\l).
\]
For any $f \in W^{t+1}(\HH)$ with decomposition
\[
f = \int_\RR^\oplus f_\l \,ds(\l)
\]
we have
\[
\norm{U\,f_\l}_t \ll \norm{f_\l}_{t+1}
\]
for $ds$-almost every $\l$.
\end{lemma}

\begin{proof}
Notice that from the point of view of another standard definition of Sobolev norm, namely,
\[
\Norm{f}_{t+1}^2 = \sum_{\{V_{i_1}, \dots, V_{i_{t+1}}\}\subset\Sl(2,\RR)}\Norm{V_{i_1}V_{i_2}\dots V_{i_{t+1}}f}^2,
\]
this lemma is obvious, and in fact $\Norm{U_1 f}_t \leq \Norm{f}_{t+1}$. This other norm is equivalent to the norm we are using, meaning that the two are asymptotic in the sense of ``$\asymp$''.  In particular, $\Norm{U_1 f}_t \ll \Norm{f}_{t+1}$, and this continues to hold in almost every component of the direct integral decomposition of $\HH$.
\end{proof}

\subsection*{Irreducible unitary representations of products} 
As hinted above, an irreducible unitary representation of the $d$-fold product $\PSL(2,\RR)^d$ is a $d$-fold tensor product $\HH_{\bnu}=\HH_1\otimes\dots\otimes\HH_d$, where each $\HH_j :=\HH_{\nu_j}$ is an i.u.r.~of $\PSL(2,\RR)$. We now have a basis 
\[
	\{u(\bk)\}_{\bk \in \bZ_{\bnu}} = \{u^{(1)}(k_1)\otimes\dots\otimes u^{(d)}(k_d)\}_{(k_1,\dots,k_d)\in\Z_{\nu_1}\times\dots\times\Z_{\nu_d}}.
\]
Again, we have nice expressions for Sobolev norms. Namely,
\[
	\norm{f}_t^2 = \sum_{\bk\in\bZ_{\bnu}} (1 + \mu_1 +\dots+\mu_d + 2\abs{\bk}^2)^t\,\abs{f(\bk)}^2\,\norm{u(\bk)}^2.
\]
It is convenient to define projected versions of elements of $\HH_{\bnu}$, for example
\[
	(f\mid_{k_j,\dots,k_d}) = \sum_{i=1}^{j-1}\sum_{k_i \in \Z_{\nu_i}} f(k_1,\dots,k_d)\,\norm{u(k_j)}\dots\norm{u(k_d)}\,u(k_1)\otimes\dots\otimes u(k_{j-1})
\]
is $f$ projected to $\HH_{\nu_1}\otimes\dots\otimes\HH_{\nu_{j-1}}$ by fixing $k_j,\dots,k_d$. Easy calculations of Sobolev norms show that
\begin{equation}\label{eqn:obsone}
\norm{(f\mid_{k_j,\dots,k_d})}_\t^2 \leq \norm{f}_\t^2
\end{equation}
and
\begin{equation}\label{eqn:obstwo}
\sum_{k_1\in\Z_{\nu_1}}(1+Q_{\nu_1}(k_1))^\t\,\norm{(f\mid_{k_1})}_\s^2 \leq \norm{f}_{\t+\s}^2.
\end{equation}
We use~\eqref{eqn:obsone} and~\eqref{eqn:obstwo} repeatedly.

\subsection*{Invariant distributions}
Let $\PSL(2,\RR)\to\U(\HH)$ be a unitary representation. A distribution is an element $\D\in\E'(\HH):=(\Cinf(\HH))^*$ of the dual to the space of smooth vectors. The distribution is $U$-invariant if $\LLL_U \D=0$. That is, if $\D(Uv)=0$ for every $v\in\Cinf(\HH)$. Similarly, we define distributions of order $s$ to be elements of $W^{-s}(\HH)$, the dual to the Sobolev space of order $s$. 

Flaminio and Forni~\cite{FF} find all the $U$-invariant distributions in irreducible unitary representations of $\PSL(2,\RR)$. First, there is $\D^+$, defined by
\[
	\D^+(u(k)) = 1 \quad\forall k \in \Z_\nu,
\]
regardless of $\nu$'s value. For discrete series representations, there are no other independent invariant distributions. But for principal and complementary series, we also have
\[
	\D^-(u(k)) = \begin{cases} \displaystyle{\prod_{i=1}^{\abs{k}}\frac{2i-1-\nu}{2i-1+\nu}} &\textrm{if }\nu\neq 0\\
									\displaystyle{\sum_{i=1}^{\abs{k}}\frac{1}{2i-1}} &\textrm{if }\nu=0,
					\end{cases}
\]
where empty products are by convention $1$ and empty sums are by convention $0$.

\begin{lemma}\label{lem:distributionsum}
There is a number $\cee>0$ such that
\[
\sum_{\pm} \sum_{m\in\Z} (1+Q_\nu (m))^{-t}\,\frac{\abs{\D^{\pm}(u(m))}^2}{\norm{u (m)}^2} \ll_{\nu_0, t} \begin{cases}
	\displaystyle{\frac{1}{(1+\mu)^{t-\cee}}} &\textrm{princ., comp.}\\ \\
	\displaystyle{\frac{1}{(1+\mu+2n^2)^{t-\cee}}} &\textrm{disc.}
	\end{cases}
\]
\end{lemma}

\begin{proof}
The proof is a calculation in each of the three families of irreducible unitary representations of $\PSL(2,\RR)$. 

\paragraph{For principal series representations,} where $\nu \in i\RR$,
\begin{multline}\label{eq:rewrite}
\sum_{\pm} \sum_{m\in\Z} (1+Q_\nu (m))^{-t}\,\frac{\abs{\D^{\pm}(u(m))}^2}{\norm{u (m)}^2} \\ = \sum_{m\in\Z} (1+Q_\nu (m))^{-t}\,\left(\abs{\D^{+}(u(m))}^2 + \abs{\D^{-}(u(m))}^2\right).
\end{multline}
If $\nu\neq 0$, this is rewritten
\begin{multline*}
\sum_{m\in\Z} (1+Q_\nu (m))^{-t}\,\left[1 + \prod_{i=1}^{\abs{m}}\Abs{\frac{2i-1-\nu}{2i-1+\nu}}^2\right] \\
= \sum_{m\in\Z} (1+Q_\nu (m))^{-t}\,\left[1 + \prod_{i=1}^{\abs{m}}\frac{(2i-1)^2+\abs{\nu}^2}{(2i-1)^2+\abs{\nu}^2}\right],
\end{multline*}
since $\nu$ is purely imaginary. Hence we only need to bound
\[
\sum_{m\in\Z} (1+Q_\nu (m))^{-t} = \sum_{m\in\Z} (1+\mu+2m^2)^{-t},
\]
which for $t>0$, can be compared to the integral $\int\frac{dx}{(1+\mu+2x^2)^t}$ to obtain the desired bound $\ll_t (1+\mu)^{1/2-t}$.

Now, if $\nu=0$,~\eqref{eq:rewrite} is rewritten
\[
\sum_{m\in\Z} (1+Q_0 (m))^{-t}\,\left[1 + \Abs{\sum_{i=1}^{\abs{m}}\frac{1}{2i-1}}^2\right].
\]
Comparing to an integral, we bound this by
\begin{multline*} 
\leq \sum_{m\in\Z} (1+Q_0 (m))^{-t}\,\left[1 + \Abs{1+\frac{1}{2}\log(2m-1)}^2\right] \\
= \sum_{m\in\Z} (1+1/4+ 2m^2)^{-t}\,\left[1 + \Abs{1+\frac{1}{2}\log(2m-1)}^2\right]
\end{multline*}
which is also bounded by $\ll_t (1+\mu)^{1/2-t}$.

\paragraph{For complementary series representations,}
where $\nu \in (-1,1)\backslash\{0\}$ we can rewrite~\eqref{eq:rewrite} as 
\begin{multline*}
\sum_{m\in\Z} (1+Q_\nu (m))^{-t}\,\left[1 + \prod_{i=1}^{\abs{m}}\Abs{\frac{2i-1-\nu}{2i-1+\nu}}^2\right] \\
= \sum_{m\in\Z} (1+Q_\nu (m))^{-t}\,\left[1 + \norm{u(m)}^4\right],
\end{multline*}
which, applying~\cite[Lemma~2.1]{FF}, is bounded by
\[
	\ll \sum_{m\in\Z} (1+Q_\nu (m))^{-t}\,\left[1+ \left(\frac{1-\nu}{1+\nu}\right)^2\,(1+m)^{-2\nu}\right],
\]
which is in turn bounded $\ll_{\nu_0,t} (1+\mu)^{1/2-t}$ where $\abs{\nu}\leq\nu_0<1$.

\paragraph{For discrete series representations,} where $\nu=2n-1$, the expression to bound becomes
\[
\sum_{m\in\Z} (1+Q_\nu (m))^{-t}\,\frac{\abs{\D^{+}(u(m))}^2}{\norm{u (m)}^2} = \sum_{m\in\NN} (1+Q_\nu (n+m))^{-t}\,\norm{u (n+m)}^{-2}.
\]
Again, we appeal to~\cite[Lemma~2.1]{FF} to bound this by
\[
\ll \sum_{m\in\NN} (1+Q_\nu (n+m))^{-t}\,\left(\frac{1}{m+1}\right)^{-\nu},
\]
which is bounded $\ll (1+\mu+2n^2)^{\cee-t}$.
\end{proof}

\section{Top degree}\label{sec:topdegree}

The strategy for the top-degree part of Theorem~\ref{thm:unipotent} is to work in an irreducible unitary representation $\HH_1\otimes\dots\otimes\HH_d := \HH_\otimes \otimes \HH_d$ and write $f$ as a sum $f_\otimes + f_d$, in such a way that $(f_\otimes \mid_\ell)\in\HH_\otimes$ is always in the kernel of all $U_1,\dots,U_{d-1}$-invariant distributions, and $(f_d \mid_{\bk})\in\HH_d$ is always in the kernel of every $U_d$-invariant distribution. This will facilitate an induction on $d$, with base case given by~\cite[Theorem~$4.1$]{FF}.

\paragraph{}

In any irreducible unitary representation $\HH_\mu$, choose $\varphi_\pm \in W^s(\HH_\mu)$ so that 
\begin{equation}\label{eqn:varphiconditions}
\begin{split}
\varphi_+ \in \ker\D_\mu^{-} \textrm{ and }\varphi_- \in \ker\D_\mu^{+} \\ 
\D^+ (\varphi_+)=1 \textrm{ and } \D^- (\varphi_-)=1.
\end{split}
\end{equation}
For example, one can easily check that the following choices satisfy~\eqref{eqn:varphiconditions}. If $\nu\neq 0$, and $\pi_\nu$ is from the principal or complementary series, 
\begin{equation}\label{eqn:varphidefinitions1}
\begin{split}
\varphi_+ (k) &= \begin{cases} \frac{\nu-1}{2\nu} &\textrm{if}\quad k=0\\ 
							\frac{\nu+1}{2\nu} &\textrm{if} \quad k=1\\
							0  &\textrm{otherwise} \end{cases} \\
\varphi_- (k) &= \begin{cases} \frac{\nu+1}{2\nu} &\textrm{if}\quad k=0\\ 
							-\frac{\nu+1}{2\nu} &\textrm{if} \quad k=1\\
							0  &\textrm{otherwise.} \end{cases} \\
\end{split}
\end{equation}
If $\nu = 0$, 
\begin{equation}\label{eqn:varphidefinitions2}
\begin{split}
\varphi_+ (k) &= \begin{cases} 1 &\textrm{if}\quad k=0\\ 
							0  &\textrm{otherwise} \end{cases} \\
\varphi_- (k) &= \begin{cases} -1 &\textrm{if}\quad k=0\\ 
							1 &\textrm{if} \quad k=1\\
							0  &\textrm{otherwise.} \end{cases} \\
\end{split}
\end{equation}
And if $\pi_\nu$ is from the discrete series, there is only one invariant distribution, $\D^+$, so we only choose $\varphi_+$ as 
\begin{equation}\label{eqn:varphidefinitions3}
\begin{split}
\varphi_+ (k) &= \begin{cases} 1 &\textrm{if}\quad k=0\\ 
							0  &\textrm{otherwise.} \end{cases} \\
\end{split}
\end{equation}
and set $\varphi_- \equiv 0$ for convenience. Bear in mind that these are not the only choices that would ``do the trick.'' We will keep $\varphi_\pm$ defined as in~\eqref{eqn:varphidefinitions1},~\eqref{eqn:varphidefinitions2},~\eqref{eqn:varphidefinitions3}, but other choices would work as well. See \S\ref{sec:discussion} for a discussion of $\varphi$'s.

\begin{lemma}\label{lem:squaresumphi}
\[ 
\sum_\pm \norm{\varphi_\pm}_t^2 \ll_{\e_0,\nu_0} \begin{cases}
									(1+\mu)^t &\textrm{in principal and complementary series}  \\
									(1+\mu+2n^2)^t &\textrm{in discrete series}
										\end{cases}
\]
\end{lemma}

\begin{proof}
Let $t>0$, and compute
\begin{align}
	\sum_{\pm}\norm{\varphi_\pm}_t^2 &= \sum_{k\in\Z_\nu} \left(1 + Q_\nu (k)\right)^t\,\left(\Abs{\varphi_+(k)}^2 + \Abs{\varphi_-(k)}^2\right)\,\norm{u(k)}^2 \nonumber \\
		\begin{split}
		&= \left(1 + Q_\nu (0)\right)^t\,\left(\Abs{\varphi_+(0)}^2 + \Abs{\varphi_-(0)}^2\right)\,\norm{u(0)}^2 \\
		&\quad+\left(1 + Q_\nu (1)\right)^t\,\left(\Abs{\varphi_+(1)}^2 + \Abs{\varphi_-(1)}^2\right)\,\norm{u(1)}^2. 
		\end{split}\label{eq:split}
\end{align}
Now we just bound in all possible cases. 

\paragraph{For principal series representations,} where $\nu\neq 0$, the expression~\eqref{eq:split} becomes
\[
\left(1 + \mu\right)^t\,\left(\Abs{\frac{\nu-1}{2\nu}}^2 + \Abs{\frac{\nu+1}{2\nu}}^2\right)
		+\left(3 + \mu\right)^t\,\left(\Abs{\frac{\nu+1}{2\nu}}^2 + \Abs{-\frac{\nu+1}{2\nu}}^2\right),
\]
and since $\nu$ is purely imaginary we can bound by
\[
\ll \left(1 + \mu\right)^t\,\frac{\abs{\nu}^2+1}{\abs{\nu}^2} \ll_{\e_0} \left(1+\mu\right)^t.
\]
This is what Assumption~\ref{ass:ume} was for.

On the other hand, if $\nu=0$,~\eqref{eq:split} becomes
\[
\left(1 + \mu\right)^t\,\left(\Abs{1}^2 + \Abs{-1}^2\right)
		+\left(3 + \mu\right)^t\,\left(\Abs{0}^2 + \Abs{1}^2\right),
\]
so the bound is obvious.

\paragraph{For complementary series representations,} where $\nu\neq 0$, the expression~\eqref{eq:split} becomes
\begin{multline*}
\left(1 + \mu\right)^t\,\left(\Abs{\frac{\nu-1}{2\nu}}^2 + \Abs{\frac{\nu+1}{2\nu}}^2\right)\\
		+\left(3 + \mu\right)^t\,\left(\Abs{\frac{\nu+1}{2\nu}}^2 + \Abs{-\frac{\nu+1}{2\nu}}^2\right)\Abs{\frac{1-\nu}{1+\nu}},
\end{multline*}
which we can bound on $\abs{\nu}\in(\e_0,\nu_0)\subset(0,1)$ by
\[
	\ll_{\e_0,\nu_0} \left(1+\mu\right)^t.
\]
Again, Assumption~\ref{ass:ume} was made for this.

\paragraph{For discrete series representations,} where $\nu = 2n-1$ for some $n$, the expression~\eqref{eq:split} becomes 
\[
\left(1 + Q_\nu (0)\right)^t = (1+\mu+2n^2)^t
\]
which is exactly what we want to bound by.
\end{proof}

\paragraph{}
We now define for $(\bk,\ell) \in \bZ_\times \times\Z_d$
\begin{align}\label{eqn:fotimes}
	f_{\otimes} (\bk,\ell) &= \sum_{\pm} \varphi_\pm (\ell)\,\sum_{m\in \Z_d} f(\bk, m)\,\D^{\pm}(u_d(m))
\end{align}
and put $f_d = f - f_\otimes$. Let us prove that $f_\otimes$ (and therefore $f_d$ also) retains some of $f$'s Sobolev regularity.

\begin{lemma}\label{lem:regularity}
There is a function $L:\RR_+\to \RR_+$ such that $$\norm{f_\otimes}_t \ll_{\nu_0, \e_0,t} \norm{f}_{L(t)}$$ for all $t>0$.
\end{lemma}

\begin{proof}
We compute
\begin{multline*}
\norm{f_\otimes}_t^2 = \sum_{\bk,\ell}(1+Q_+ (\bk) + Q_d (\ell))^t\,\abs{f_\otimes (\bk,\ell)}^2\,\norm{u_\otimes (\bk)\otimes u_d (\ell)}^2 \\
	= \sum_{\bk,\ell}(1+Q_+ (\bk) + Q_d (\ell))^t\,\Abs{\sum_{\{+,-\}} \varphi_\pm (\ell) \sum_{m\in\Z_d} f(\bk,m)\D^{\pm}(u_d(m))}^2\,\norm{u_\otimes (\bk)\otimes u_d (\ell)}^2
\end{multline*}
and by repeated use of the Cauchy--Schwartz Inequality,
\begin{multline*}
	\leq \sum_{\bk,\ell}(1+Q_+ (\bk) + Q_d (\ell))^t \sum_{\pm} \Abs{\varphi_\pm (\ell)}^2\,\norm{u_d (\ell)}^2\\ 
	\times \sum_{\pm}\sum_{m\in\Z_d} (1+Q_d (m))^{t'}\,\abs{f(\bk,m)}^2\,\norm{u_\otimes (\bk)\otimes u_d (m)}^2 \sum_{m\in\Z_d} (1+Q_d (m))^{-t'}\,\frac{\abs{\D^{\pm}(u_d(m))}^2}{\,\norm{u_d (m)}^2}.
\end{multline*}
Re-writing this using~\eqref{eqn:obstwo}, we obtain
\[
	\leq \norm{f}_{t+t'}^2 \,\sum_{\pm} \norm{\varphi_\pm}_t^2 \, \sum_{\pm} \sum_{m\in\Z_d} (1+Q_d (m))^{-t'}\,\frac{\abs{\D^{\pm}(u_d(m))}^2}{\,\norm{u_d (m)}^2}.
\]
If we choose $t'=t+\cee$, Lemmas~\ref{lem:squaresumphi} and~\ref{lem:distributionsum} imply 
\[
	\ll_{\nu_0,\e_0,t} \norm{f}_{2t+\cee}^2,
\]
which is the lemma, with $L(t)=2t+\cee$.
\end{proof}

\begin{lemma}\label{lem:kernels}
Suppose $f \in\ker\II_{U_1,\dots,U_d}(\HH_\otimes \otimes\HH_d)$. Then for every $\ell \in \Z_d$, we have $(f_\otimes\mid_\ell)\in\ker\II_{U_1,\dots,U_{d-1}}(\HH_\otimes)$ and for every $\bk\in\bZ_\times$, we have $(f_d \mid_{\bk})\in\ker\II_{U_d}(\HH_d)$.
\end{lemma}

\begin{proof}
This is essentially~\cite[Lemma~4.3]{Ramhc} and~\cite[Lemma~13.3]{Ramhcii}, but since our $f_\otimes$ is now defined differently, it is worth reproducing the proof for this new scenario.

We calculate, for $\bullet$ a multi-index of length $d-1$ consisting of $+$'s and $-$'s,
\begin{align*}
	\D^{\bullet} (f_\otimes\mid_\ell) &= \sum_{\bk \in \bZ_\times}f_\otimes(\bk,\ell)\,\D^\bullet (u_\otimes (\bk)) \\
		&= \sum_{\bk \in \bZ_\times}\left[\sum_{\pm} \varphi_\pm (\ell)\,\sum_{m\in \Z_d} f(\bk, m)\,\D^{\pm}(u_d(m))\right]\,\D^\bullet (u_\otimes (\bk)) \\
		&= \sum_{\pm} \varphi_\pm (\ell)\,\left(\D^{\bullet, +}(f) + \D^{\bullet, -}(f)\right) = 0,
\end{align*}
and for $\bullet \in \{+,-\}$,
\begin{align*}
	\D^\bullet (f_d \mid_{\bk}) &= \sum_{\ell \in\Z_d}\left[f(\bk,\ell) - f_\otimes (\bk,\ell)\right]\D^\bullet (u_d (\ell)) \\
		&= \D^\bullet (f\mid_{\bk}) - \sum_\pm \D^\bullet(\varphi_\pm)\sum_{m\in\Z_d}f(\bk,m)\,\D^\pm (u_d(m)) \\
		&=0,
\end{align*}
because $\D^\bullet (\varphi_\pm)= \d_{\bullet,\pm}$.
\end{proof}

\begin{theorem}[Version of Theorem~\ref{thm:sobtop} for i.u.r.s]\label{thm:topiurs}
Suppose $f \in \Cinf(\HH_1\otimes\dots\otimes\HH_d)$ lies in the kernel of all $U_1,\dots,U_d$-invariant distributions. Then there exists a solution $g_1,\dots,g_d \in \Cinf(\HH_1\otimes\dots\otimes\HH_d)$ to the degree-$d$ coboundary equation for $f$ that satisfies the Sobolev estimates $\norm{g_i}_t \ll_{\nu_0, \e_0, t} \norm{f}_{\s_d}$, where $\s_d := \s_d (t)$ is some increasing function of $t>0$.
\end{theorem}

\begin{proof}
The proof is an induction on $d$, with the base case being~\cite[Theorem~4.1]{FF}.

By Lemma~\ref{lem:kernels} we have that $(f_\otimes\mid_\ell)\in\ker\II_{U_1,\dots,U_{d-1}}(\HH_\otimes)$ for all $\ell \in \Z_d$ and $(f_d \mid_{\bk})\in\ker\II_{U_d}(\HH_d)$ for every $\bk\in\bZ_\times$, so the inductive assumption provides $g_{1,\ell},\dots,g_{d-1,\ell} \in \Cinf(\HH_\otimes)$ and $g_{d,\bk}\in\Cinf(\HH_d)$ satisfying $U_1\,g_{1,\ell}+\dots+U_{d-1}\,g_{d-1,\ell} = (f_\otimes\mid_\ell)$ and $U_d\,g_{d,\bk} = (f_d \mid_{\bk})$ and the bounds $\norm{g_{i,\ell}}_t \ll_{\nu_0,t} \norm{(f_\otimes\mid_\ell)}_{\s_{d-1}}$ and $\norm{g_{d,\bk}}_t \ll_{\nu_0,t} \norm{(f_d\mid_{\bk})}_{\s_1}$. By putting $g_i(\bk,\ell):= g_{i,\ell}(\bk)$ for $i=1,\dots,d-1$ and $g_d(\bk,\ell) := g_{d,\bk}(\ell)$, we define a solution $g_1,\dots,g_d \in \Cinf(\HH_\otimes \otimes\HH_d)$ to the degree-$d$ coboundary equation for $f$. To see the bounds on Sobolev norms, we calculate
\begin{align*}
\norm{g_i}_t^2 &= \sum_{\bk,\ell}(1 + Q_+ (\bk) + Q_d(\ell))^t\abs{g_i(\bk,\ell)}^2 \norm{u_\otimes (\bk)\otimes u_d(\ell)}^2 \\
	&\leq \sum_{\bk,\ell}(1 + Q_+ (\bk))^t(1 + Q_d(\ell))^t\abs{g_i(\bk,\ell)}^2 \norm{u_\otimes (\bk)}^2\norm{u_d(\ell)}^2 \\
	&= \sum_{\ell}(1 + Q_d(\ell))^t\,\norm{g_{i,\ell}}_t^2 \,\norm{u_d(\ell)}^2 \\
	&\ll_{\nu_0,t} \sum_{\ell}(1 + Q_d(\ell))^t\,\norm{(f_\otimes\mid_\ell)}_{\s_{d-1}}^2 \,\norm{u_d(\ell)}^2 \\
	&\ll_{\nu_0,t} \norm{f_\otimes}_{\s_{d-1}+t}^2 \ll_{\nu_0,\e_0, t} \norm{f}_{L(\s_{d-1}+t)},
\end{align*}
by Lemma~\ref{lem:regularity}. A very similar computation holds for Sobolev norms of $g_d$. The theorem is proved by setting $\s_d(t) = L(\s_{d-1}(t) + t)$.
\end{proof}


\section{Lower degrees}
Let $\O_{\RR^d}^{n}(W^s (\HH_1\otimes\dots\otimes\HH_d))$ denote the set of leafwise $n$-forms defined by elements of Sobolev order $s$.

For an $n$-form $\o\in\O_{\RR^d}^n (W^s (\HH_1\otimes\dots\otimes\HH_d))$, we define $\o_\ell$ for $\ell\in\{1,\dots,d\}$ to be the part of $\o$ that forgets $\ell$:
\[
	\o_\ell (U_{i_1},\dots,U_{i_n}) = \o (U_{i_1},\dots,U_{i_n})
\]
where $i_1<\dots<i_n \subset \{1,\dots,\widehat\ell,\dots,d\}$. If we now fix some basis element $u_\ell (k)\in \HH_\ell$, we can define a restriction 
\[
	(\o_\ell\mid_k) \in \O_{\RR^{d-1}}^n(W^s(\HH_1\otimes\dots\otimes\widehat{\HH_\ell}\otimes\dots\otimes\HH_d))
\]
by setting
\[
	(\o_\ell\mid_k)(U_{i_1}, \dots, U_{i_n}) = (\o (U_{i_1}, \dots, U_{i_n}))\mid_k
\]	
for $i_1<\dots<i_n \subset \{1,\dots,\widehat\ell,\dots,d\}$.

\begin{lemma}\label{lem:closed}
Let $\o\in\O_{\RR^d}^n (W^s (\HH_1\otimes\dots\otimes\HH_d))$ with $\di\o=0$. Then for any $\ell = 1,\dots,d$, we have that $\di(\o_\ell\mid_k)=0$ for all $k \in\Z_\ell$.
\end{lemma}

\begin{proof}
This is \cite[Lemma~14.1]{Ramhcii}.
\end{proof}

\begin{proposition}\label{prop:hrt}
Let $\o\in\O_{\RR^d}^{d-1} (W^s (\HH_1\otimes\dots\otimes\HH_d))$ be a closed $(d-1)$-form, and $\ell=1,\dots,d$. Then for every $k\in \Z_\ell$, we have that
\[
	(\o_\ell\mid_k)(U_1,\dots,\widehat{U_\ell}, \dots, U_d) \in \ker\II_{U_1,\dots,\widehat{U_\ell},\dots,U_d}^s (\HH_{\otimes,\widehat\ell}).
\]
\end{proposition}

\begin{proof}
\cite[Lemma~5.2]{Ramhc}
\end{proof}

\begin{theorem}[Version of Theorem~\ref{thm:soblower} for i.u.r.s]\label{thm:loweriurs}
Let $\HH=\HH_1\otimes\dots\otimes\HH_d$ be an irreducible representation of $\PSL(2,\RR)^d$ with no trivial factor, and let $1\leq n\leq d-1$. Then any smooth $n$-cocycle $\o\in\O_{\RR^d}^n(\Cinf(\HH))$ is a coboundary $\di\eta=\o$ for some $\eta\in\O_{\RR^d}^{n-1}(\Cinf(\HH))$ and $\norm{\eta}_t \ll_{\nu_0,\e_0, t}\,\norm{\o}_{\vars_d}$ where $\vars_d:=\vars_d(t)$ is some increasing function of $t>0$.
\end{theorem}

\begin{proof}
This proof is a version of the proof of~\cite[Theorem~14.6]{Ramhcii}, which in turn is adapted from an induction in~\cite[p. 25]{KK95}. For us, the base case is~\cite[Theorem~17]{M1}, for $1$-cocycles an $\RR^2$-action. The inductive strategy is to suppose we have the result for $\RR^p$-actions, whenever $2\leq p \leq d-1$. Let $t>0$.

For any $k \in Z_1$, $(\o_1\mid_k)$ is a closed $n$-form over the $\RR^{d-1}\cong\langle U_2,\dots,U_d\rangle$-action on $\HH_{\otimes,\widehat1}$, by Lemma~\ref{lem:closed}. 

If $n<d-1$, then our induction assumption produces an $(n-1)$-primitive $\eta_{1,k}\in\O_{\RR^{d-1}}^n(\Cinf(\HH_{\otimes,\widehat 1}))$ for $(\o_1\mid_k)$ satisfying 
\begin{equation}\label{eq:etavars}
	\norm{\eta_{1,k}}_\t \ll_{\nu_0, \e_0, \t} \norm{(\o_1\mid_k)}_{\vars_{d-1}(\t)}
\end{equation}
for any $\t>0$. On the other hand, if $n=d-1$, then $(\o_1\mid_k)$ is a top-degree form for the $\RR^{d-1}\cong\langle U_2,\dots,U_d\rangle$-action, and our Proposition~\ref{prop:hrt} tells us that $(\o_1\mid_k)(U_2,\dots,U_d)\in\ker_{U_2,\dots,U_d}(\HH_{\otimes,\widehat1})$. Theorem~\ref{thm:topiurs} now tells us that there is an $(n-1)$-primitive for $(\o_1\mid_k)$ which we will again call $\eta_{1,k}$, and that satisfies 
\begin{equation}\label{eq:etasig}
	\norm{\eta_{1,k}}_\t \ll_{\nu_0, \e_0, \t} \norm{(\o_1\mid_k)}_{\s_{d-1}(\t)}
\end{equation}
for any $\t>0$. Now we just define $\eta_1$ by the requirement that $(\eta_1\mid_k) = \eta_{1,k}$ for every $k\in \Z_1$. 

It is left to find a primitive for the components of $\o$ which contain the index $1$. Let
\[
\theta(U_{i_2},\dots,U_{i_n}) = \o(U_1, U_{i_2},\dots,U_{i_n}) - U_1\,\eta_1(U_{i_2},\dots,U_{i_n})
\]
and notice that $\theta \in \O_{\RR^d}^{n-1}(\Cinf(\HH))$ and that 
\[
\Norm{\theta(U_{i_2},\dots,U_{i_n})}_\t \leq \Norm{\o(U_1, U_{i_2},\dots,U_{i_n})}_\t + \Norm{U_1\,\eta_1(U_{i_2},\dots,U_{i_n})}_\t.
\]
By Lemma~\ref{lem:diffnorm}, 
\begin{align*}
\Norm{\theta(U_{i_2},\dots,U_{i_n})}_\t &\leq \Norm{\o(U_1, U_{i_2},\dots,U_{i_n})}_\t + \Norm{\eta_1(U_{i_2},\dots,U_{i_n})}_{\t+1} \\ 
	&\ll_{\nu_0, \t} \norm{\o}_{\vars_{d}(\t+1)}
\end{align*}
and this in turn implies that 
\begin{equation}\label{eq:theta}
	\norm{\theta}_\t \ll_{\nu_0, \t} \norm{\o}_{\vars_d (\t+1)}.
\end{equation}

The next calculation shows that
\[
	(\theta_1\mid_k) = (\theta\mid_k) \in \O_{\RR^{d-1}}^{n-1}(\Cinf(\HH_{\otimes,\widehat1}))
\]
is a closed form for any $k\in \Z_1$. For $1<i_1<\dots<i_n\leq d$, 
\begin{align*}
\di(\theta\mid_k)(U_{i_1},\dots,U_{i_n}) &= \sum_{j=1}^{n}(-1)^{j+1}U_{i_j}\,(\theta\mid_k)(U_{i_1},\dots,\widehat U_{i_j},\dots,U_{i_n})\\
		&=\sum_{j=1}^{n}(-1)^{j+1}U_{i_j}\,\left(\o(U_1,U_{i_1},\dots,\widehat U_{i_j},\dots,U_{i_n})\mid_k\right)\\
		&\quad -\sum_{j=1}^{n}(-1)^{j+1}U_{i_j}\,\left(U_1\,\eta_1(U_{i_1},\dots,\widehat U_{i_j},\dots,U_{i_n})\mid_k\right) \\
		&= U_1\,(\o\mid_k)(U_{i_1},\dots,U_{i_n}) \tag{because $\di\o=0$}\\
		&\quad - U_1\sum_{j=1}^{n}(-1)^{j+1}U_{i_j}\,(\eta_1\mid_k)(U_{i_1},\dots,\widehat U_{i_j},\dots, U_{i_n}) \\
		&=0. \tag{because $\di(\eta_1\mid_k)=(\o_1\mid_k)$}
\end{align*}
Therefore, our induction implies that there exists a primitive $\zeta_k \in \O_{\RR^{d-1}}^{n-2}(\Cinf(\HH_{\otimes,\widehat1}))$ for $(\theta_1\mid_k)$ satisfying 
\begin{equation}\label{eq:zeta}
	\norm{\zeta_k}_\t \ll_{\nu_0,\e_0, \t} \norm{(\theta_1\mid_k)}_{\vars_{d-1}(\t)}
\end{equation}
for every $\t>0$. Finally, we define $\eta$ by
\begin{equation*}
	(\eta\mid_k)(U_{i_1},\dots, U_{i_{n-1}}) = \begin{cases}\zeta_k (U_{i_2},\dots, U_{i_{n-1}}) &\textrm{if } i_1 = 1 \\
																		(\eta_1\mid_k)(U_{i_1},\dots, U_{i_{n-1}}) &\textrm{if } i_1 >1. \end{cases}
\end{equation*}
Then $\di\eta=\o$, so we have found our primitive. Let us check the Sobolev estimates by calculating
\[
	\norm{\eta}_t^2 = \sum_{1\leq i_1< \dots <i_n\leq d} \norm{\eta(U_{i_1},\dots, U_{i_{n-1}})}_t^2.
\]
We consider each term separately. 

On one hand, if $i_1 = 1$, then
\begin{align*}
	\norm{\eta(U_1,U_{i_2},\dots,U_{i_{n-1}})}_t^2 &= \sum_{(k,\bell)\in\Z_1\times\bZ_\times}\left(1+Q_1(k)+Q_+(\bell)\right)^t \\
		&\quad\times\abs{\eta(U_1,U_{i_2},\dots,U_{i_{n-1}})(k,\bell)}^2\,\norm{u(k)\otimes v(\bell)}^2 \\
		&\leq \sum_{k\in\Z_1}\left(1+Q_1(k)\right)^t \sum_{\bell\in\bZ_\times}\left(1+Q_+(\bell)\right)^t \\
		&\quad\times\abs{\eta(U_1,U_{i_2},\dots,U_{i_{n-1}})(k,\bell)}^2\,\norm{u(k)}^2\,\norm{v(\bell)}^2 \\
		&= \sum_{k\in\Z_1}\left(1+Q_1(k)\right)^t\,\norm{(\eta\mid_k)(U_1,U_{i_2},\dots,U_{i_{n-1}})}_t^2 \\
		&= \sum_{k\in\Z_1}\left(1+Q_1(k)\right)^t\,\norm{\zeta_k(U_1,U_{i_2},\dots,U_{i_{n-1}})}_t^2 \\
		&\ll_{\nu_0, \e_0, t} \sum_{k\in\Z_1}\left(1+Q_1(k)\right)^t\,\norm{(\theta\mid_k)}_{\vars_{d-1}(t)}^2 \tag{by~\eqref{eq:zeta}} \\
		&\leq \norm{\theta}_{\vars_{d-1}(t)+t}^2 \\
		&\ll_{\nu_0,\e_0, t} \norm{\o}_{\vars_{d-1}(\vars_{d-1}(t)+t+1)}^2. \tag{by~\eqref{eq:theta}}
\end{align*}
On the other hand, if $i_1>1$, then
\begin{align*}
	\norm{\eta(U_{i_1},\dots,U_{i_{n-1}})}_t^2 &= \sum_{(k,\bell)\in\Z_1\times\bZ_\times}\left(1+Q_1(k)+Q_+(\bell)\right)^t \\
		&\quad\times\abs{\eta(U_{i_1},\dots,U_{i_{n-1}})(k,\bell)}^2\,\norm{u(k)\otimes v(\bell)}^2 \\
		&\leq \sum_{k\in\Z_1}\left(1+Q_1(k)\right)^t \sum_{\bell\in\bZ_\times}\left(1+Q_+(\bell)\right)^t \\
		&\quad\times\abs{\eta_1(U_{i_1},\dots,U_{i_{n-1}})(k,\bell)}^2\,\norm{u(k)}^2\,\norm{v(\bell)}^2 \\
		&= \sum_{k\in\Z_1}\left(1+Q_1(k)\right)^t\,\norm{(\eta_1\mid_k)(U_{i_1},\dots,U_{i_{n-1}})}_t^2 \\
		&\ll_{\nu_0, \e_0,t} \sum_{k\in\Z_1}\left(1+Q_1(k)\right)^t\,\norm{(\o_1\mid_k)}_{\max\{\vars_{d-1}(t), \s_{d-1}(t)\}}^2 \tag{by~\eqref{eq:etavars} and~\eqref{eq:etasig}}\\
		&\leq \norm{\o_1}_{\max\{\vars_{d-1}(t), \s_{d-1}(t)\}+t}^2 \leq \norm{\o}_{\max\{\vars_{d-1}(t), \s_{d-1}(t)\}+t}^2
\end{align*}
These two calculations imply that there is some increasing function $\vars_d(t)$ such that $\norm{\eta}_t\ll_{\nu_0, \e_0,t}\norm{\o}_{\vars_d(t)}$ holds. (For example,
\[
	\vars_d (t) = \max\left\{\begin{matrix}\vars_{d-1}(\vars_{d-1}(t)+t+1), \\ \vars_{d-1}(t)+t \\ \s_{d-1}(t)+t \end{matrix}\right\},
\]
works.)
\end{proof}

\section{Proofs}

The proofs of Theorems~\ref{thm:sobtop} and~\ref{thm:soblower} now follow from Theorems~\ref{thm:topiurs} and~\ref{thm:loweriurs} by arguments identical to the proofs of~\cite[Theorems~10.1 and~10.2]{Ramhcii}. Theorem~\ref{thm:unipotent} follows trivially from Theorems~\ref{thm:sobtop} and~\ref{thm:soblower}. Theorem~\ref{thm:mixed} is proved by mixing the arguments in this note with those in~\cite{Ramhc}.

\begin{proof}[Proof of Theorem~\ref{thm:sobtop}]
Let $\HH$ be a unitary representation of $\SL(2,\RR)^d$ with spectral gap and satisfying Assumption~\ref{ass:ume}. This means we can choose $\nu_0,\e_0$ uniformly over the direct integral decomposition
\begin{equation}\label{eq:decomp}
	\HH = \int_\RR^\oplus \HH_{\bnu_{\l}}\,ds(\l)
\end{equation}
where $ds$-almost every $\HH_\l:=\HH_{\bnu_\l}$ is an irreducible unitary representation. (Sobolev spaces $W^s(\HH)$ also decompose accordingly.) Now $f \in \ker\II_{U_1,\dots,U_d}(\HH)$ decomposes as
\[
	f = \int_\RR^\oplus f_\l\,ds(\l)
\]
where $f_\l \in \ker\II_{U}(\HH_\l)$ for $ds$-almost every $\l$. Therefore, Theorem~\ref{thm:topiurs} guarantees that for $ds$-almost every $\l$ there are $g_{1,\l},\dots,g_{d,\l}\in\Cinf(\HH_\l)$ satisfying
\[
	U_1\,g_{1,\l}+\dots+U_d\,g_{d,\l}=f
\]
and the estimates $\norm{g_{i,\l}}_t\ll_{\nu_0, \e_0,t} \norm{f_\l}_{\s_d(t)}$, where $\nu_{i,\l}\leq\nu_0<1$ for all $\nu_\l\in\CC$ appearing in the decomposition~\eqref{eq:decomp}. Setting
\[
	g_i = \int_\RR^\oplus g_{i,\l}\,ds(\l),
\]
we have a solution to the degree-$d$ coboundary equation for $f$, satisfying the estimate $\norm{g_i}_t \ll_{\nu_0,\e_0,t}\norm{f}_{\s_d(t)}$, proving the theorem.
\end{proof}


\begin{proof}[Proof of Theorem~\ref{thm:soblower}]
Let $\HH, \nu_0,\e_0, t$ be as in the theorem statement, and $1 \leq n \leq d-1$.  Let 
\[
	\o \in \O_{\RR^d}^{n}(\Cinf(\HH))\quad\textrm{with}\quad \di\o = 0.
\]
Again, we have a direct integral decompostion
\[
	\Cinf(\HH) = \int_{\RR}^\oplus \Cinf(\HH_{\bnu_\l})\,ds(\l)
\]
where $ds$-almost every $\HH_{\bnu_\l}:=\HH_\l$ is irreducible and without trivial factors (by assumption), and $\o$ decomposes
\[
	\o(U_{i_1}, \dots, U_{i_n}) = \int_\RR^{\oplus} \o_{\l}(U_{i_1}, \dots, U_{i_n})\,ds(\l)
\]
such that $ds$-almost every $\o_\l$ is a cocycle in $\O_{\RR^d}^{n}(\Cinf(\HH_{\l}))$.

For these $\l$, Theorem~\ref{thm:loweriurs} supplies $\eta_\l:=\eta_{\bnu_\l} \in \O_{\RR^d}^{n-1}(\Cinf(\HH_{\bnu_\l}))$ with $\di\eta_{\l} = \o_{\l}$ and
\[  
	\norm{\eta_\l}_t \ll_{\nu_0,\e_0,t} \norm{\o}_{\varsigma_d(t)}.
\]
Defining
\[
	\eta(U_{i_1}, \dots, U_{i_{n-1}}) := \int_\RR^{\oplus} \eta_{\l}(U_{i_1}, \dots, U_{i_{n-1}})\,ds(\l),
\]
gives a solution to the coboundary equation $\di\eta=\o$ satisfying the bound $\norm{\eta}_{t}\ll_{\nu_0,\e_0,t}\norm{\o}_{\vars_d(t)}$.
\end{proof}


\section{Discussion}\label{sec:discussion}

This is the third paper where this general strategy has been implemented, each time with significant adjustments. However, we can interpret the previous efforts in terms of the procedure used here. The strategies in~\cite{Ramhc}, where we treated Anosov $\RR^d$-actions on
\[
\SL(2,\RR)^d/\G = \SL(2,\RR)\times\dots\times\SL(2,\RR)/\G, 
\]
and~\cite{Ramhcii}, where we considered Weyl chamber flows associated to $d$-fold products
\[
\SO^\circ(N,1)^d = \SO^\circ(N,1)\times\dots\times\SO^\circ(N,1), 
\]
correspond to making particular choices of $\varphi$'s in their respective situations. Therefore, the innovation in this note has been the observation that the arguments from those articles work if we re-define $f_\otimes$ in terms of these $\varphi$'s, as we did here in~\eqref{eqn:fotimes}.

In principle, the recipe should apply more generally for $\RR^d$-actions on irreducible homogeneous spaces of $G_1\times\dots\times G_d$ where $G_i$ are (semisimple) Lie groups with the following ingredients:

\paragraph{A nice description of the invariant distributions in irreducible unitary representations of $G$, with spanning set $\{\D^w\}_{w\in\W}$ indexed by some set $\W$.}
\begin{itemize}
\item In the $\SL(2,\RR)\times\dots\times\SL(2,\RR)$ case, we used work of Mieczkowski on geodesic flows of hyperbolic surfaces~\cite{M2}. He showed that in any irreducible unitary representation of $\PSL(2,\RR)$, the space of $\left(\begin{smallmatrix}1/2 & 0 \\ 0 & -1/2\end{smallmatrix}\right)$-invariant distributions is at most two-dimensional, and explicitly gave a spanning set $\{\D^0, \D^1\}$ in each irreducible. The methods followed those of Flaminio--Forni for horocycle flows~\cite{FF}.
\item For hyperbolic manifolds of arbitrary dimension, we proved~\cite[Theorem~1.2]{Ramhcii}, a result showing that contrary to the surfaces case, the invariant distributions for the geodesic flow element form an \emph{infinite}-dimensional space in any irreducible unitary representation of $\SO^\circ(N,1)$, when $N\geq 3$. Like in Mieczkowski and Flaminio--Forni's work, the proof proceeded by looking at the action of the geodesic flow element on $K$-types in irreducible unitary representations, and relating invariant distributions to a (partial) difference equation with a combinatorial flavor. We found a spanning set $\{\D^{\bm,\l}\}_{(\bm,\l)\in M\times \L}$, indexed by certain Gelfand--Cejtlin arrays. (See~\cite[Figure~1]{Ramhcii}.)
\item In this article, we use the seminal work of Flaminio--Forni~\cite{FF}, where the space of horocycle flow-invariant distributions is studied. As in the case of geodesic flows of surfaces, it turns out that the space of invariant distributions is at most two-dimensional in any irreducible unitary representation of $\PSL(2,\RR)$.
\end{itemize}

\paragraph{A way to choose $\varphi$'s that satisfy~\eqref{eqn:varphiconditions} and some sort of controlling statement like Lemma~\ref{lem:squaresumphi}.} 
\begin{itemize}
\item Mieczkowski gives the geodesic flow-invariant distributions in any irreducible unitary representation of $\PSL(2,\RR)$. This is a space of (at most) two dimensions, and he labels a spanning set by $\{\D^0, \D^1\}$. If we were to follow the strategy used in this paper, the results in~\cite{Ramhc} could be proved by defining elements $\varphi_0$ and $\varphi_1$ satisfying~\eqref{eqn:varphiconditions}. Actually, we can interpret the strategy in that paper as a \emph{family} of choices of $\varphi_0$'s and $\varphi_1$'s, which reflects the fact that primitives to higher-degree coboundary equations are not unique, as they are for the first degree.
\item For geodesic flows of higher-dimensional hyperbolic manifolds, since we have an infinite-dimensional space of invariant distributions in every irreducible unitary representation of $\SO^\circ(N,1)$, we would have $\{\varphi_{\bm,\l}\}$ where $\bm,\l$ correspond to certain Gelfand--Cejtlin arrays. Part of the challenge therefore is in coping with this infinite-dimensionality. Carrying the arguments from the $\SL(2,\RR)$ situation directly over to this case is impossible. Therefore, instead of making the ``family'' of choices that we had made in the previous paper, we simply chose the ``easiest'' $\varphi$'s: Each $\varphi_{\bm,\l}$ was a basis element of the irreducible unitary representation, where the basis $\{u_{\bm,\l}\}$ was exactly what had been used to find our basis $\{\D^{\bm,\l}\}$ of invariant distributions in the first place.
\item In this paper, $\varphi_\pm$ are indexed by $\{+,-\}$, the indices for the spanning set $\{\D^+,\D^-\}$ of invariant distributions found in~\cite{FF}, and are engineered to satisfy~\eqref{eqn:varphiconditions}. 
\end{itemize}

\paragraph{A ``base case'' for the induction telling us that we indeed have a complete set of obstructions to the degree-$1$ coboundary equation and giving us control over the Sobolev norms of primitives.} 
\begin{itemize}
\item For Anosov $\RR^d$-actions on $\SL(2,\RR)^d/\G$ we used Mieczkowski's~\cite[Theorem~4.3]{M2}, a theorem stating that one can solve the coboundary equation for the geodesic flow as long as the given function is in the kernel of all geodesic flow-invariant distributions, and giving estimates on the Sobolev norms of primitives. (Actually, Mieczkowski proved this for $\PSL(2,\RR)$, but the extension to $\SL(2,\RR)$ is not hard.) 
\item For Weyl chamber flows of $\SO^\circ(N,1)^d$, we had to establish a base case~\cite[Theorem~1.1]{Ramhcii} for the geodesic flow of a hyperbolic manifold. The result was a statement analogous to Mieczkowski's. This was the main concern of Part~I of that article.
\item In fact, both of the above mentioned theorems were inspired by the work of Flaminio and Forni, where a similar analysis was carried out for horocycle flows of surfaces, resulting in \cite[Theorem~4.1]{FF}, which has been our ``base case'' here.
\end{itemize}

\paragraph{}
Once these ingredients are present, it should be possible to carry out an induction like the one used here for the top-degree theorem (see the proof of Theorem~\ref{thm:topiurs}). The lower-degree statement follows comparatively easily after the top-degree part is settled, by using the induction used here in the proof of Theorem~\ref{thm:loweriurs}, which is itself adapted from~\cite{KK95}.

\section*{Acknowledgments}
This work was completed while the author was employed at the University of Bristol and supported by ERC. Much of it was written in October--December 2012, while visiting the Department of Mathematics at the Pennsylvania State University and enjoying their hospitality.


\bibliographystyle{amsalpha}
\bibliography{../bibliography}
\end{document}